\documentclass[11pt,twoside]{article}

\usepackage{graphicx}
\usepackage{multirow}
\usepackage{fancyhdr}
\usepackage{array}
\usepackage{amsfonts}
\usepackage{amsmath}
\usepackage{amssymb}
\usepackage{amsthm}

\usepackage{subfigure}

\usepackage{graphics} 
\usepackage{epsfig} 
\usepackage{latexsym}

\usepackage[textwidth=16cm,textheight=22cm,top=3cm,bottom=3cm]{geometry}
\usepackage{fancyhdr}

\newtheorem{theorem}{Theorem}
\newtheorem{lemma}{Lemma}
\newtheorem{corollary}{Corollary}

\newtheorem{remark}{Remark}
\theoremstyle{definition}

\usepackage{color}

\newcommand{\drop}[1]{}
\newcommand{\no}{\noindent}
\newcommand{\fer}[1]{(\ref{#1})}
\newcommand{\qtext}[1]{\quad\text{#1}}

\newcommand{\bA}{\mathbf{A}}
\newcommand{\bb}{\mathbf{b}}

\newcommand{\bef}{\mathbf{f}}
\newcommand{\bG}{\mathbf{G}}
\newcommand{\bn}{\mathbf{n}}
\newcommand{\bp}{\mathbf{p}}

\newcommand{\bs}{\mathbf{s}}
\newcommand{\bT}{\mathbf{T}}
\newcommand{\bu}{\mathbf{u}}
\newcommand{\bv}{\mathbf{v}}
\newcommand{\bw}{\mathbf{w}}
\newcommand{\bx}{\mathbf{x}}
\newcommand{\by}{\mathbf{y}}
\newcommand{\bz}{\mathbf{z}}

\newcommand{\bpsi}{\boldsymbol{\psi}}
\newcommand{\bxi}{\boldsymbol{\xi}}

\newcommand{\eps}{\varepsilon}
\newcommand{\vfi}{\varphi}

\newcommand{\grad}{\nabla}

\newcommand{\p}{\partial}

\newcommand{\N}{\mathbb{N}}
\newcommand{\R}{\mathbb{R}}

\def\O{\Omega}

\newcommand{\e}{{\text{e}}}

\newcommand{\abs}[1]{| #1 |}
\newcommand{\nor}[1]{\| #1 \|}

\DeclareMathOperator{\sign}{sign}
\DeclareMathOperator{\supp}{supp}

\title{ Well-posedness of  a cross-diffusion population model with nonlocal diffusion
\thanks{Supported by Spanish MCI Project MTM2017-87162-P.}}

\author{Gonzalo Galiano  \thanks{Dpt. of Mathematics, Universidad de Oviedo,
 c/ Calvo Sotelo, 33007-Oviedo, Spain ({\tt galiano@uniovi.es, julian@uniovi.es})}
    \and Juli\'an Velasco\footnotemark[2] }

\date{}

\begin{document}

\thispagestyle{plain}

\maketitle

\begin{abstract}
  We prove the existence and uniqueness of solution of a nonlocal cross-diffusion competitive population model for two species. The model may be considered as a version, or even an approximation, of the paradigmatic Shigesada-Kawasaki-Teramoto cross-diffusion model, in which the usual diffusion differential operator is replaced by an integral diffusion operator. The proof of existence of solutions is based on a compactness argument, while the uniqueness of solution is achieved through a duality technique.  
  
\no\emph{Keywords: }
nonlocal diffusion, cross-diffusion, evolution problem, existence of solutions, uniqueness of solution, Shigesada-Kawasaki-Teramoto population model

\no\emph{AMS Subject Classifications: } 
  35R09, 45K05, 92D25
  
\end{abstract}

\section{Introduction}
Let $T>0$ and $\O\subset\R^d$ $(d\geq 1)$ be an open and bounded set with Lipschitz continuous boundary.  We consider the following problem. For $i=1,2$, find $u_i:[0,T]\times\O\to\R_+$ such that 
\begin{align}
&  \p_t u_i(t,\bx)  =  \int_\O J(\bx-\by) \big(p_i (\bu(t,\by))-p_i (\bu(t,\bx)) \big) d\by +f_i(\bu(t,\bx)), \label{eq.eq}\\
&   u_i(0,\bx)=  u_{0i}(\bx),  \label{eq.id}
\end{align}
 for $(t,\bx)\in Q_T=(0,T)\times\O$, and for some $u_{0i}:\O\to\R_+$. Here, $\R_+=[0,\infty)$,  $\bu=(u_1,u_2)$, the diffusion kernel, $J:\R^d\to\R_+$, is an even function  and, for $i,j=1,2$, $i\neq j$, the diffusion and reaction functions are given by
  \begin{align}
  \label{def:fun}
 p_i(\bu) = u_i(c_i+a_i u_i + u_j),\quad  f_i(\bu) = u_i\big(\alpha_i-(\beta_{i1} u_1 + \beta_{i2}u_2)\big),
\end{align}
for some non-negative constant coefficients $c_i,~a_i,~\alpha_i,~\beta_{ij}$.

Problem \fer{eq.eq}-\fer{eq.id} with diffusion and reaction functions given by \fer{def:fun} is a nonlocal diffusion version of the Shigesada-Kawasaki-Teramoto (SKT) population model introduced in \cite{Shigesada1979}, which reads, for $i=1,2$,
\begin{align}
&  \p_t v_i  =  \Delta p_i (\bv) +f_i(\bv)&&\text{in }Q_T, & \label{eq.eqs}\\
&  \grad p_i (\bv) \cdot \bn = 0 &&\text{on }(0,T)\times\p\O, & \label{eq.bcs}\\
&   v_i(0,\cdot)=  u_{0i}  && \text{in }\O. \label{eq.ids}
\end{align}

The SKT problem \fer{eq.eqs}-\fer{eq.ids} has attracted much attention in the last decades due to several factors, among which its capacity of producing non-uniform steady states, of capturing population segregation phenomena, or of exhibiting instability with respect to the uniform steady states leading to pattern formation.
None of these properties are verified if the diffusion functions, $p_i$, lack of the cross terms $u_1u_2$.

In addition, the mathematical theory developed to prove the well-posedness of the model is quite sophisticated, mainly due to the fact that cross-diffusion systems of PDE's do not enjoy, in general, of a comparison principle allowing to employ classical techniques such as the method of sub- and super-solutions. Moreover, no maximum or minimum principles hold, so that even the non-negativity of the solution components is not evident.

The literature on the SKT problem is abundant.  
Regarding the problem of existence of weak solutions, the first global existence result is due to Kim \cite{Kim1984}, for a simplified version of the problem ($a_i=0$, one space dimension). Yagi \cite{Yagi1993} deduced that if the self-diffusion coefficients are small ($8a_i >1$), a global weak solution do exist, the smallness condition implying that the diffusion matrix is positive definite, and hence the problem is uniformly parabolic.  This result was extended in \cite{Galiano2003a} to the coefficient restriction $a_i>0$. In this case, the diffusion matrix is not, a priori, definite positive, and entropy estimates obtained by using the test functions $\ln(u_i)$ play a key role to overcome the difficulty of obtaining suitable gradient estimates. 
The result in \cite{Galiano2003a},  holding for one-dimensional spatial domains, was extended by 
Chen and J\"ungel \cite{Chen2004}  to up to three-dimensional domains. In a generalization of their techniques, J\"ungel \cite{Jungel2015boundedness} showed, among other properties, that the domain dimension may be arbitrarily taken, and further generalized the form of the diffusion functions. This generalization had been already studied by Desvillettes et al. \cite{Desvillettes2014entropy},
who also contributed to the understanding of the triangular system (when one of the $u_1u_2$ cross-diffusion terms are absent in the equations) \cite{Desvillettes2015new}, extending the particular results obtained by Amann \cite{Amann1989} from his general theory on quasilinear parabolic systems.

Regarding the problem of uniqueness of solutions of the SKT problem,  
Amann \cite{Amann1989} proved the result in the triangular case. In \cite{Galiano2012}, uniqueness of the full system is proven for weak solutions under the assumption $\grad u_i\in L^\infty(Q_T)$. More recently, Chen and J\"ungel \cite{Chen2018preprint} have proven the weak-strong uniqueness property for renormalized solutions under several parameter restrictions. That is, given a renormalized solution $\bu$ of \fer{eq.eqs}-\fer{eq.ids},  if a strong solution $\tilde\bu$, with $\p_t \tilde u_i,~\grad \tilde u_i\in L^\infty(Q_T)$, does exist then $\bu = \tilde \bu$.
However, the uniqueness of a weak solution of the full SKT problem in the same functional space in which existence is  proven remains an open problem.

Efforts have been also pointed out in other directions: the existence of global classical solutions, see e.g.  \cite{Le2006}, the existence of non-uniform steady states, e.g. \cite{Lou1999}, or the onset of instabilities from perturbations of uniform steady states leading to pattern formation \cite{Gambino2013}, among others. 

\bigskip

To motivate terming problem \fer{eq.eq}-\fer{eq.id} as a \emph{nonlocal diffusion version} of the SKT problem \fer{eq.eqs}-\fer{eq.ids} let us consider the following example, introduced and analyzed by Andreu et al. \cite{Andreu2010book}. This example shows that the Neumann problem for the heat equation
\begin{align}
&  \p_t v  =  \Delta v &&\text{in }Q_T, & \label{eq.eqa}\\
&  \grad v \cdot \bn = 0 &&\text{on }(0,T)\times\p\O, & \label{eq.bca}\\
&   v(0,\cdot)=  v_{0}  && \text{in }\O, \label{eq.ida}
\end{align}
may be approximated by nonlocal diffusion problems of the type
\begin{align}
&  \p_t u(t,\bx)  =  \int_\O J(\bx-\by) (u(t,\by)-u(t,\bx)) d\by, \label{eq.eqan}\\
&   u(0,\bx)=  u_{0}(\bx),  \label{eq.idan}
\end{align}
for $(t,\bx)\in Q_T$, under an appropriate rescaling of the diffusion kernel, which we assume here to be radially symmetric. Indeed, defining 
 \begin{align}
 \label{def.jeps}
  J_\delta(\bz) = \frac{c_1}{\delta^{2+d}}J\Big(\frac{\bz}{\delta}\Big),\qtext{with } c_1^{-1} = \frac{1}{2} \int_{\R^d} J(\bz)z_d^2 d\bz,
 \end{align}
it is proven that the sequence $u_\delta$ obtained as solutions of \fer{eq.eqan}-\fer{eq.idan}
with $J$ replaced by $J_\delta$ is such that 
\begin{align*}
 \lim_{\delta\to0}\nor{u_\delta-v}_{L^\infty(Q_T)}=0,
\end{align*}
where $v$ is the solution of the heat problem \fer{eq.eqa}-\fer{eq.ida}. Similar results are obtained for nonlinear heat equations or $p-$Laplacian diffusion operators, see \cite{Andreu2010book}. 

A formal argument justifying this convergence result is easy to describe in the one-dimensional setting. Consider a smooth function,  $u$,  and the integral operator
\begin{align*}
 A_\delta(u)(x) = \int_\R J_\delta(x-y) (u(y)-u(x)) dy.
\end{align*}
Introducing the change $y=x-\delta z$ and using the Taylor's expansion of $u$ in powers of $\delta$, we get 
\begin{align*}
 A_\delta(u)(x) = \frac{c_1}{\delta}\int_\R J(z)zdz u'(x) + \frac{c_1}{2}\int_\R J(z)z^2dz u''(x) + O(\delta).
\end{align*}
Since $J$ is even, the first term of the right hand side vanishes, so we deduce
\begin{align*}
 A_\delta(u)(x) \to u''(x) \qtext{as }\delta\to0.
\end{align*}
A similar formal argument applies to \fer{eq.eq}-\fer{eq.id}, and in this sense we interpret that \fer{eq.eq}-\fer{eq.id} is a nonlocal diffusion version (or approximation) of  the SKT  original problem \fer{eq.eqs}-\fer{eq.ids}.

The theory developed by Andreu et al. to tackle the problem of existence of solutions to nonlinear versions of the nonlocal diffusion problem \fer{eq.eqan}-\fer{eq.idan} 
is mainly based on semi-group theory and strongly relies on the monotonicity of the nonlocal diffusion operator. However, non-monotone diffusion functions appear often in applications, specially those arising in image processing. For instance, the image restoration bilateral filter \cite{Smith1997, Tomasi1998, Buades2005}, in its continuous evolution formulation, takes the form
\begin{align}
\label{BF}
\p_t u(t,\bx) = \int_\O \exp\Big(-\frac{\abs{\bx-\by}^2}{\rho^2}\Big) \exp\Big(-\frac{\abs{u(t,\bx)-u(t,\by)}^2}{h^2}\Big) (u(t,\by)-u(t,\bx))d\by,
\end{align}
for $(t,\bx)\in Q_T$, where $\O$ is the space of pixels, $u(0,\cdot)$ is the image to be filtered, and $\rho$ and $h$ are positive constants modulating the space and range neighborhoods where the filtering process takes place. 
 
 Due to the lack of monotonicity of the integral operator in \fer{BF} with respect to $u$, the theory developed in \cite{Andreu2010book} is not applicable to this problem. In \cite{Galiano2019}, we introduced a compactness argument to show the existence of global solutions of a general class of problems including \fer{BF}. Our proof is based 
on obtaining suitable estimates of the gradient of the solution by differentiating equation \fer{BF}. Assuming enough regularity on the kernel and diffusion functions, the gradient estimate only depends on the $L^\infty(Q_T)$ boundedness of the solution.   
In the scalar case, this bound is obtained as a consequence of the kernel and diffusion functions symmetry, implying a comparison principle.

Extending this idea to systems of equations, in particular to the SKT problem, relies again in obtaining suitable $L^\infty(Q_T)$ estimates of the solution components which provide, after differentiation of \fer{eq.eq}, estimates of their gradients too, leading to the compactness of an appropriate sequence of approximating functions. 

Recall that the $L^\infty(Q_T)$ boundedness  of solutions of the local diffusion SKT model \fer{eq.eqa}-\fer{eq.ids} has not been proved \cite{Jungel2015boundedness}, fact that introduces serious difficulties in the analisys of this problem. In the local diffusion case, the compactness argument is based on introducing the Lyapunov functional, also known as \emph{entropy functional},
\begin{align}
\label{def.ent}
 E(t) = \sum_{i=1}^2 \int_\O \big(u_i(\ln(u_i)-1)+1) \geq0,
\end{align}
and, by formally using  $\ln(u_i)$ as a test function in \fer{eq.eqa}, deduce the following entropy and gradient estimates \cite{Galiano2003a, Chen2004}
\begin{align*}
 E(t)+ \sum_{i=1}^2 a_i \int_{Q_t} \abs{\grad u_i}^2 \leq E(0)+c. 
\end{align*}
Interestingly, in the nonlocal diffusion problem the entropy functional plays also an important role, in this case for obtaining the $L^\infty(Q_T)$ boundedness of the solution. The formal argument is the following. Assuming the (non-trivial) property $u_i>0$ in $Q_T$, and integrating \fer{eq.eq} in $(0,t)$ for $t<T$, we obtain
\begin{align}
\label{est.eleinfint}
u_i(t,\bx)  \leq 
u_{0i}(\bx) & + C \nor{J}_{L^\infty} \big( \nor{u_i}_{L^1}+ \nor{u_i}_{L^2}^2+\nor{u_1}_{L^2}\nor{u_2}_{L^2}\big) \\
& +\alpha_i \int_0^t u_i(\tau,\bx)d\tau. \nonumber
\end{align}
Thus, if $L^1(Q_T)$ and $L^2(Q_T)$ estimates of $u_i$ are provided, and if $u_{0i}\in L^\infty(\O)$, then Gronwall's lemma implies $u_i\in L^\infty(Q_T)$. The $L^1(Q_T)$ estimate of $u_i$ is obtained by direct integration of \fer{eq.eq} in $\O$. The $L^2(Q_T)$ estimate of $u_i$ is also trivial if $\beta_{ii}>0$, and deduced by integration of $\fer{eq.eq}$ in $Q_T$. However, if $\beta_{ii}=0$ (and $a_i>0$) we must resort to using the test function $\ln(u_i)$ to obtain the following entropy and $L^2(Q_T)$ estimate of $u_i$
\begin{align*}
 E(t)+ \sum_{i=1}^2 a_i \int_{Q_t} \abs{u_i}^2 \leq E(0)+c. 
\end{align*}
We thus see that the result of testing the differential equations of the local and nonlocal diffusion problems with $\ln(u_i)$ both lead to the compactness of an appropriate sequence of approximating solutions. For the local diffusion problem, due to direct estimation of the gradients. For the nonlocal diffusion problem, due to the estimation of the $L^\infty(Q_T)$ norms which yield, thanks to the Lipschitz continuity of the diffusion and reaction functions, the gradient estimates. 

Of course, the previous estimations are just formal because the possibility of $u_i$ vanishing in some subset of $Q_T$ may not be overridden. The aim of this article is
giving conditions on the data and formulating an approximating scheme which lead to proving the existence of solutions of \fer{eq.eq}-\fer{eq.id}. The $L^\infty(Q_T)$ regularity of the resulting solutions is the main tool to prove that, in fact, there exists a unique solution. 

\begin{remark}
 Although we motivated why the solution of the nonlocal SKT problem may be viewed as an approximation to the local diffusion SKT problem, we can not expect the  $L^\infty(Q_T)$ bound of the former to be transferred to the latter. Indeed, \fer{est.eleinfint} shows that the $L^\infty(Q_T)$ bound for the nonlocal problem depends on the $L^\infty(\O)$ bound of the kernel function, $J$. Since the  nonlocal-local diffusion approximation procedure depends on the introduction of a singular kernel, $J_\delta$, see \fer{def.jeps}, the corresponding $L^\infty(Q_T)$ bound of the sequence of solutions of the nonlocal problem (approximating to the local diffusion problem) will, in general, blow up as $\delta\to0$.
\end{remark}

The organization of the paper is the following. In Section~\ref{sec:main} we state the assumptions on the data which ensure the existence and uniqueness of solutions of problem \fer{eq.eq}-\fer{eq.id}, and formulate our main result. In Section~\ref{sec:app} we solve an approximated and regularized problem for which we are able to obtain suitable uniform entropy and $L^\infty(Q_T)$ estimates of its solutions.  In Section~\ref{sec:eps} we pass to the limit in the regularizing-approximating parameter, proving the existence of solutions of problem \fer{eq.eq}-\fer{eq.id}. Finally, in Section~\ref{sec:uniqueness} we prove the uniqueness of solution.

\section{Assumptions and main results}\label{sec:main}  

Since $\O\subset\R^d$ is bounded, we have $\bx-\by \in B$ for all $\bx,\by\in\O$,
 for some open ball $B\subset\R^d$ centered at the origin. Thus, for $J$ defined on $\R^d$, we may 
 always replace it in \fer{eq.eq} by its restriction to $B$, $J|_B$. Abusing on notation, we write $J$ instead of  
$J|_B$ in the rest of the paper.

We always assume, at least, the following hypothesis on the data.
  
\no \textbf{Assumptions (H)}
\begin{enumerate}
\item The final time, $T>0$, is arbitrarily fixed. The spatial domain, $\O\subset\R^d$ $(d\geq 1)$, is an open and bounded set with Lipschitz continuous boundary.

 \item The kernel function $J \in L^\infty(B)\cap BV(B)$ is even and non-negative, with 
 \begin{align}
 \label{cond.j1}
 \{\bx\in B: \nor{\bx}\leq \rho \}\subset \supp(J), 
 \end{align} 
 for some positive constant $\rho$. 
 
\item The initial data $u_{0i}\in L^\infty(\O)\cap BV(\O)$ are non-negative, for $i=1,2$. 

\item For $i,j=1,2$, $i\neq j$, the constants $c_i,~a_i,~\alpha_i,~\beta_{ij}$ are non-negative.
\end{enumerate}

 
In the following theorem we state the main result of this article. There are some important differences  in the results for the local and nonlocal diffusion models. On one hand, nonlocal diffusion operators do not produce a spatial regularization effect on the solution with respect to the initial data \cite{Andreu2010book}. Thus, since it does not provide compactness, the diffusion operator does not play an essential role for the existence of solutions of the model. This is reflected in the possibility of allowing the linear and self-diffusion coefficients to vanish. That is, the case $c_i=a_i=0$ is not excluded (if $\beta_{ii}>0$) for the nonlocal diffusion model. However, such case is certainly excluded in the local diffusion model.   

On the other hand, in the local diffusion model the initial data may be taken from a large space of distributions, being the corresponding notion of solution interpreted in the weak sense. Our result for the nonlocal diffusion problem assumes $u_{0i}\in L^\infty(\O)\cap BV(\O)$ and returns a strong solution. While the $L^\infty(Q_T)$ boundedness of the initial data  is a common assumption in reaction-diffusion systems, the bounded variation is a technical assumption needed to give sense to  the spatial differentiation of \fer{eq.eq}. However, notice that  the BV regularity is an usual standard in image processing problems like \fer{BF} and that, nonetheless,    scalar   problems with monotone diffusion functions only need of $L^1(\O)$ regularity of the initial data \cite{Andreu2010book, Galiano2019}.

\begin{theorem}
 \label{th.existence}
Assume (H) and
\begin{align*}
 a_i+\beta_{ii} >0, \qtext{for } i=1,2 .
\end{align*}
Then, there exists a unique strong solution $(u_1,u_2)$ of problem \fer{eq.eq}-\fer{eq.id} with $u_i\geq 0$ a.e. in $Q_T$ and such that, 
for $i=1,2$ and $t\in [0,T]$,
\begin{align}
&  u_i\in W^{1,\infty}(0,T;L^\infty(\O))\cap C([0,T];L^\infty(\O)\cap BV(\O)), \nonumber\\
& E(t) + \sum_{i=1}^2 a_i \int_{Q_t}\int_\O J(\bx-\by) \big(u_i(s,\bx)-u_i(s,\by)\big)^2 d\by d\bx ds \leq E(0)+c, \label{th1:entropy}
\end{align}
with $E(t)$ defined by \fer{def.ent}, and for some constant $c>0$ independent of $J$.
\end{theorem}

\begin{remark}

\begin{enumerate}
\item The notion of strong solution of \fer{eq.eq}-\fer{eq.id} is the usual: a function $\bu$ with $u_i\in W^{1,1}(0,T;L^1(\O)) \cap C([0,T]; L^2(\O)$ satisfying the equations in the a.e. sense in $Q_T$.

\item It is a common assumption to impose the normalizing condition $\int_{\R^d} J(\by)d \by = 1$, implying 
\begin{align}
 \label{conc.J2}
  \int_{\R^d} J(\by-\bx)d \by = 1 \qtext{for a.e. } \bx\in\O.
\end{align}
However, this property is no longer true if the integration is performed in $\O$. 
Condition \fer{cond.j1} and the Lipschitz continuity of $\p\O$, implying the interior cone property, allows to keep 
a property weaker than \fer{conc.J2} but enough to our purposes. Defining $m:\O\to\R_+$ by $m(\bx)= \int_{\O} J(\bx-\by)d \by, $
we have, for some constant $J_0>0$,  
 \begin{align}
 \label{conc.J}
  J_0 \leq m(\bx) \leq \nor{J}_{L^1(B)} \qtext{for a.e. }\bx\in\O.
 \end{align}

\end{enumerate}

 \end{remark}

\section{Existence of solutions of a regularized and approximated problem}\label{sec:app}

Let $\eps\in(0,1)$ and consider two sequences of functions $J_\eps$ and $u_{0\eps i}$ satisfying (H) and, in addition, 
\begin{align}
\label{ass:regularity}
J_\eps\in W^{1,1}(B) ,\quad  u_{0\eps i}\in W^{1,\infty}(\O).
\end{align}
We may construct these sequences  to have, as $\eps\to0$, 
\begin{align*}
  & J_\eps \to J \qtext{strongly in }L^q(B), \text{ with }\nor{J_\eps}_{L^\infty(B)}\leq K, &&&  \\
 & u_{0\eps i} \to u_{0i} \qtext{strongly in }L^q(\O), \text{ with } \nor{u_{0\eps i}}_{L^\infty(\O)}\leq K, &&&  
 \end{align*}
for any $q\in[1,\infty)$, where $K>0$ is independent of $\eps$, and 
\begin{align*}
 \nor{\nabla J_{\eps}}_{L^1(B)}\to \text{TV}(J),\quad  \nor{\nabla u_{0\eps i}}_{L^1(\O)}\to \text{TV}(u_{0i}), 
\end{align*}
where TV denotes the total variation with respect to the $\bx$ variable, see \cite{Ambrosio2000}. 
Notice that, in particular,
\begin{align}
& \grad J_{\eps}  \qtext{is uniformly bounded in }L^1(B) , \label{propeps2}\\
& \grad u_{0\eps i}  \qtext{is uniformly bounded in }L^1(\O) , \label{propeps1b} 
\end{align}
and that the function 
\begin{align}
 \label{conc.Jeps}
  m_\eps(\bx)= \int_{\O} J_\eps(\bx-\by)d \by, 
\end{align}
may be taken satisfying property \fer{conc.J}, possibly redefining the $\eps$- independent constant $J_0>0$. More in general, and using the $L^1(\O)$ uniform boundedness of $J_\eps$, we deduce that $m_\eps$ satisfies 
 \begin{align}
 \label{conc.Jeps2}
  J_0 \leq m_\eps(\bx) \leq J_1 \qtext{for a.e. }\bx\in\O,
 \end{align}
for some positive constants $J_0,~J_1$ independent of $\eps$.
 
In this section, we prove the existence of solutions of the following approximated and regularized problem. For $i=1,2$, find $u_i:[0,T)\times\O\to\R$ such that, for $(t,\bx)\in Q_T$,  
\begin{align*}
& \p_t u_i(t,\bx) =  \int_\O J_\eps(\bx-\by) \big(p_i (\bu^+(t,\by)+\eps)-p_i (\bu^+(t,\bx)+\eps) \big) d\by + f_i(\bu^+(t,\bx)+\eps),\\
& u_{i}(0,\bx) = u_{0\eps i}(\bx),
\end{align*}
where we used the notation $v_i = v_i^+-v_i^-$ for splitting a scalar function into its positive and negative parts, and write $\bv^+=(v_1^+,v_2^+)$.  We also denote by (L) the following straightforward property: For $i=1,2$,
\begin{align*}
p_i,~f_i\text{ and the positive part are Lipschitz continuous functions.}\qquad \text{(L)}
\end{align*}

\subsection{Existence of solutions of a time independent problem}

Let  $N\in\N$, $M_0 = \max_{i=1,2}\nor{u_{0\eps i}}_{L^\infty}\leq K$, and set $M_{j}=M_0\sum_{k=0}^j 2^{-k}$ for $j=0,1,\ldots, N$, so that $M_j \leq 2M_0$ for all $j$. Consider the collection of complete metric spaces
\begin{align*}
V_{j}=\{ \bv \in W^{1,\infty}(\O)\times W^{1,\infty}(\O): \nor{v_i}_{L^\infty} \leq M_j,\text{ for }i=1,2\}.
\end{align*}
 Let $\bu_\eps^0 = \bu_{0\eps}$. For $j=0,1,\ldots,N-1$, assume that $\bu^j_\eps \in V_{j}$ is given and consider the operator $\bT^{j+1}$ defined on $V_{j+1}$ by, for $i=1,2$,  
\begin{align}
\label{fpop}
 T_i^{j+1}(\bv)(\bx) = u_{\eps i}^j(\bx) & + \tau_{j+1}  \int_\O J_\eps(\bx-\by) \big(p_i (\bv^+(\by)+\eps)-p_i (\bv^+(\bx)+\eps) \big) d\by \\ 
 & + \tau_{j+1} f_i(\bv^+(\bx)+\eps), \nonumber
\end{align}
where $\tau_{j+1}>0$ is a constant to be fixed.

Let us check that $\bT^{j+1}$ has a fixed point in $V_{j+1}$.  To do this, we employ the Banach's fixed point theorem. 

First notice that \fer{ass:regularity} and (L) imply  
$\bT^{j+1}(V_{j+1}) \subset W^{1,\infty}(\O)\times W^{1,\infty}(\O)$.
Using that $J_\eps$ is uniformly bounded in $L^\infty(B)$ and the explicit expressions of $p_i$ and $f_i$,  we obtain
\begin{align*}
 \abs{T_i^{j+1}(\bv)(\bx)} & \leq M_{j}+C_0\tau_{j+1} (1+M_{j+1}+M_{j+1}^2),
\end{align*}
where $C_0$ is a constant independent of $j$ and $\eps$. Taking into account that $M_0\leq M_{j}\leq 2M_0$ for all $j$ and choosing 
\begin{align*}
\tau_{j+1} < \frac{C(M_0)}{2^{j+1}} ,
\end{align*}
with $C(M_0)\leq M_0/(C_0(1+2M_0+4M_0^2))$, 
we deduce $\bT^{j+1}(V_{j+1})\subset V_{j+1}$.

To prove the contractivity, let $\bv,\bw \in V_{j+1}$. Then 
\begin{align*}
 T_i^{j+1}(\bv)(\bx) & - T_i^{j+1}(\bw)(\bx) = \tau_{j+1}  \int_\O J_\eps(\bx-\by) \big(p_i (\bv^+(\by)+\eps)-p_i (\bw^+(\by)+\eps) \big) d\by \\
 & - \tau_{j+1} m_\eps(\bx) \big(p_i (\bv^+(\bx)+\eps)- p_i (\bw^+(\bx)+\eps) \big)     \\
 & +\tau_{j+1}\big(f_i(\bv^+(\bx)+\eps)-f_i(\bw^+(\bx)+\eps)\big).
\end{align*}
Using \fer{conc.Jeps2}, (L) and  the uniform boundedness of $M_j$, we deduce
\begin{align*}
\sum_{i=1}^2\abs{ T_i^{j+1}(\bv)(\bx) & - T_i^{j+1}(\bw)(\bx)} \leq C_1\tau_{j+1} (L_p+L_f) \nor{\bv-\bw}_{L^\infty},
\end{align*}
where $C_1$ is a constant independent of $j$ and $\eps$, and with $L_p$ and $L_f$ denoting the Lipschitz continuity constants of $\mathbf{p}$ and $\mathbf{f}$ in the interval $[-2M_0,2M_0]$.
Choosing 
\begin{align}
\label{cond.tau}
\tau_{j+1}< \min\Big(\frac{C(M_0)}{2^{j+1}}, \frac{1}{C_1(L_p+L_f)}\Big),
\end{align}
 we find that $\bT^{j+1}$ is a strict contraction in $V_{j+1}$, and therefore there exists a unique fixed point of $\bT^{j+1}$ in $V_{j+1}$, that we denote by $\bu^{j+1}_\eps$. To simplify the notation we write in the following $\bu,~\bu^j,~\tau$ instead of $\bu^{j+1}_\eps,~\bu^j_\eps,~\tau_{j+1}$, respectively. Observe that $\bu$ satisfies, for $\bx\in\O$,
\begin{align}
 u_i(\bx) = u_i^j(\bx) & + \tau  \int_\O J_\eps(\bx-\by) \big(p_i (\bu^+(\by)+\eps)-p_i (\bu^+(\bx)+\eps) \big) d\by  \label{eq:discu}\\ 
 & + \tau f_i(\bu^+(\bx)+\eps),\nonumber
\end{align}
and that $\tau$ is independent of $\eps$.
\begin{remark}
Since $\sum_j\tau_j \leq C(M_0)$, if we construct a solution of problem \fer{eq.eq}-\fer{eq.id} interpolating in time from the sequence of solutions of \fer{eq:discu}, the final time can not be arbitrarily large. That is, the solution will be a solution local in time. However, we shall obtain a posteriori estimates on  $\bu^{j+1}_\eps$ which will allow to continue the solution to any arbitrarily fixed final time.
\end{remark}

\begin{lemma}\label{lemma:1}
Let $(u_1,u_2)\in W^{1,\infty}(\O)\times W^{1,\infty}(\O)$ be given by \fer{eq:discu}. Then, for $i=1,2$ and for some positive constant $c$, independent of $\eps$ and $\tau$,  the following estimates hold:
\begin{align}
& \sum_{i=1}^2 \big(\nor{u_i^+}_{L^1} +\tau \beta_{ii}\nor{u_i^+}_{L^2}^2 \big) \leq  
\sum_{i=1}^2 \big(\nor{(u_i^j)^+}_{L^1} +c\tau \nor{u_i^+}_{L^1} \big)  +c\tau\eps, \label{lema1:ele1mas}\\
& \sum_{i=1}^2 \nor{u_i^-}_{L^1}  \leq  
\sum_{i=1}^2 \big(\nor{(u_i^j)^-}_{L^1} +c\tau\eps (1+\nor{u_i^+}_{L^1}) \big) , \label{lema1:ele1menos}\\
& \sum_{i=1}^2 \nor{u_i}_{L^\infty}  \leq \sum_{i=1}^2 \big(\nor{u_i^j}_{L^\infty} + c\tau \big( \nor{u_i^+}_{L^\infty}+\nor{u_i^+}_{L^1} + \nor{u_i^+}_{L^2}^2 \big) \label{lema1:eleinf} \\
& \qquad +c\tau\eps(1+\nor{u_i^+}_{L^\infty})\big), \nonumber\\
& \sum_{i=1}^2 \Big( E_i  -\ln(\eps) \nor{u_i^-}_{L^1} +\tau a_i \int_\O\int_\O J_\eps(\bx-\by)  \big(u_i^+(\by)- u_i^+(\bx)\big)^2 d\by d\bx  \Big)   \label{lema1:entropy} \\
& \qquad \leq 
\sum_{i=1}^2 \Big( E_i^j - \ln(\eps) \nor{(u_i^j)^-}_{L^1} +c\tau \big(E_i+  \nor{u_i^+}_{L^1} + \eps\big)\Big) , \nonumber\\
& \sum_{i=1}^2 \nor{\grad u_i}_{L^\infty} \leq \sum_{i=1}^2 \nor{\grad u_i^j}_{L^\infty}  + c\tau L(\nor{\bu}_{L^\infty}) \big( 1 + \sum_{i=1}^2 \nor{\grad u_i}_{L^\infty}\big), \label{lema1:derivada}
\end{align}
where $L(\nor{\bu}_{L^\infty})$ is the maximum of the Lipschitz continuity constants of $\bp$ and $\bef$ in $\{\bs\in\R^2: \abs{s_i}\leq \nor{u_i}_{L^\infty}\}$, and where we introduced the notation 
\begin{align*}
E_i^j = \int_\O ((u_i^j)^+(\bx)+\eps)\big(\ln((u_i^j)^+(\bx)+\eps)-1) d\bx.
\end{align*}
In particular, \fer{lema1:entropy} implies 
\begin{align}
 & \sum_{i=1}^2 \Big( E_i  - \ln(\eps) \nor{u_i^-}_{L^1} +2\tau a_i J_0\nor{u_i^+}_{L^2}^2  \Big) \leq 
\sum_{i=1}^2 \Big( E_i^j + \ln(\eps) \nor{(u_i^j)^-}_{L^1}  \label{lema1:entropy2} \\
& \qquad +c\tau \big(E_i+ a_i\nor{u_i^+}_{L^1}^2 + \nor{u_i^+}_{L^1} + \eps\big)\Big). \nonumber
\end{align}

\end{lemma}

\bigskip

\begin{proof}
$\bullet$ \emph{$L^1(Q_T)$ estimates. }
Integrating the first equation of \fer{eq:discu} in $\O$ and using the symmetry of $J_\eps$, we obtain
\begin{align}
\int_\O u_1^+(\bx)d\bx & +\tau\beta_{11} \int_\O \abs{u_1^+(\bx)}^2d\bx \leq \int_\O u_1^-(\bx)d\bx + \int_\O u_1^j(\bx)d\bx \label{est.ele1}\\ 
& + \tau \alpha_1 \int_\O u_1^+(\bx) d\bx + \tau\eps\alpha_1\abs{\O}.  \nonumber
\end{align}
Integrating the first equation of \fer{eq:discu} in $\{u_1< 0\}$, we get
\begin{align*}
-\int_\O u_1^-(\bx)d\bx  = &\int_{u_1< 0} u_1^j(\bx)d\bx \\
& + \tau \int_{u_1< 0}\int_\O J_\eps(\bx-\by)\big(p_1(\bu^+(\by)+\eps)-p_1((\eps,u_2^+(\bx)+\eps))\big)d\by d\bx 
\\ & + \tau \int_{u_1< 0} f_1((\eps,u_2^+(\bx)+\eps))d\bx.
\end{align*}
Therefore, using the explicit expressions of $p_1$ and $f_1$, we deduce
\begin{align}
\int_\O u_1^-(\bx)d\bx  \leq & -\int_{u_1< 0} u_1^j(\bx)d\bx 
+\tau \int_{u_1< 0}\int_\O J_\eps(\bx-\by)p_1((\eps,u_2^+(\bx)+\eps))d\by d\bx \nonumber\\ 
& - \tau \int_{u_1< 0} f_1((\eps,u_2^+(\bx)+\eps))d\bx \nonumber \\
& \leq  \int_\O (u_1^j)^-(\bx)d\bx 
+\tau \eps \int_\O\int_\O J_\eps(\bx-\by)(c_1+a_1\eps+u_2^+(\bx)+\eps) d\by d\bx \nonumber \\ 
& - \tau\eps \int_{u_1< 0} \big(\alpha_1 -\eps(\beta_{11}+\beta_{12}) -\beta_{12}u_2^+(\bx)\big)d\bx \nonumber \\
& \leq \int_\O (u_1^j)^-(\bx)d\bx  +\tau \eps J_1 \Big( (c_1+\eps(1+a_1))\abs{\O} + \int_\O u_2^+(\bx) d\bx \Big) \nonumber \\ 
& + \tau\eps \Big( \eps(\beta_{11}+\beta_{12})\abs{\O} + \beta_{12}\int_\O u_2^+(\bx) d\bx \Big) \nonumber \\
& \leq \int_\O (u_1^j)^-(\bx)d\bx  +c \tau \eps \Big( 1 + \int_\O u_2^+(\bx) d\bx \Big) . \label{lema1:menos}
\end{align}
Replacing \fer{lema1:menos}  in \fer{est.ele1} yields
\begin{align}
\int_\O u_1^+(\bx)d\bx  +\tau\beta_{11} \int_\O \abs{u_1^+(\bx)}^2d\bx \leq & 
\int_\O (u_1^j)^+(\bx)d\bx  + \tau \alpha_1 \int_\O u_1^+(\bx) d\bx   \label{lema1:mas}\\ 
& +c \tau \eps \Big( 1 + \int_\O u_2^+(\bx) d\bx \Big) . \nonumber
\end{align}
Estimates similar to \fer{lema1:menos} and \fer{lema1:mas} are obtained from the second equation ($i=2$) of \fer{eq:discu}, leading to \fer{lema1:ele1mas} and \fer{lema1:ele1menos}. 


\no$\bullet$ \emph{$L^\infty(\O)$ estimate. }
On one hand, if $\bx\in\{\by\in\O:u_1(\by)<0\}$ we deduce from  \fer{eq:discu}
\begin{align*}
u_1^-(\bx) = &u_1^j(\bx) +\tau\int_\O J_\eps(\bx-\by)\big(p_1(\bu^+(\by)+\eps)-p_1(\eps,u_2(\bx)^++\eps)\big) d\by \\
& + \tau f_1(\eps,u_2(\bx)^++\eps).
\end{align*}
Hence, 
\begin{align*}
u_1^-(\bx) & \leq -u_1^j(\bx) +\tau\eps J_1 \big(c_1+a_1\eps +\eps (u_2(\bx)^++\eps)\big) + \tau\eps \big(\beta_{11}\eps +\beta_{12}(u_2(\bx)^++\eps)\big)\\
& \leq (u_1^j)^-(\bx) +c\tau\eps(1+u_2^+(\bx)).
\end{align*}
On the other hand, if $\bx\in\{\by\in\O:u_1(\by)\geq0\}$ then \fer{eq:discu} yields
\begin{align*}
u_1^+(\bx) & \leq u_1^j(\bx) + \tau\int_\O J_\eps(\bx-\by)\big(p_1(\bu^+(\by)+\eps) d\by +\tau\alpha_1 (u_1^+(\bx)+\eps),
\end{align*}
implying
\begin{align*}
u_1^+(\bx)  \leq & (u_1^j)^+(\bx) + c\tau\nor{J_\eps}_{L^\infty} \big( \nor{u_1^+}_{L^1} + \nor{u_1^+}_{L^2}^2 +
\nor{u_1^+}_{L^2} \nor{u_2^+}_{L^2} \big)\\
& +c\tau (u_1^+(\bx) +\eps). 
\end{align*}
Therefore, for any $\bx\in\O$, and recalling that $J_\eps$ is uniformly bounded in $L^\infty(B)$ and that $\abs{v(\bx)} =v^+(\bx) + v^-(\bx)$, we deduce 
\begin{align*}
\abs{u_1(\bx)}  \leq & \abs{u_1^j(\bx)} + c\tau \big( \abs{u_1^+(\bx)}+\nor{u_1^+}_{L^1} + \nor{u_1^+}_{L^2}^2 +
\nor{u_1^+}_{L^2} \nor{u_2^+}_{L^2} \big)\\
& +c\tau\eps(1+u_2^+(\bx)). 
\end{align*}
A similar estimates may be obtained for $\abs{u_2(\bx)}$, leading to \fer{lema1:eleinf}.

\no$\bullet$ \emph{Entropy estimate. }
We multiply \fer{eq:discu} by $ \ln(u_i^++\eps)$ and integrate in $\O$, obtaining
\begin{align}
 \int_\O  u_i(\bx)  \ln(u_i^+(\bx)+\eps) &  d\bx =  \int_\O u_{i}^j(\bx)  \ln(u_i^+(\bx)+\eps) d\bx \label{logid} \\
 & -  \frac{\tau}{2}  \int_\O J_\eps(\bx-\by) \big(p_i (\bu^+(\by)+\eps)-p_i (\bu^+(\bx)+\eps) \big) \nonumber \\
 & \qquad \times \big(\ln(u_i^+(\by)+\eps) - \ln(u_i^+(\bx)+\eps) \big) d\by \nonumber \\
 & + \tau \int_\O   f_i(\bu^+ (\bx)+ \eps))  \ln(u_i^+(\bx)+\eps) d\bx. \nonumber
\end{align}
We now estimate the different terms of \fer{logid}.

\medskip

\no$\star$ \emph{The discrete time derivative. } Like in \cite{Chen2004}[(2.15)], 
we deduce
\begin{align}
 \int_\O (u_i(\bx)-u_{i}^j(\bx)) & \ln(u_i^+(\bx)+\eps) d\bx \label{entropy1}\\
 & \geq  E_i^{j+1}-E_i^j - \ln(\eps) \int_\O (  u_i^-(\bx) - (u_i^j)^-(\bx) )d\bx. \nonumber
\end{align}

\medskip

\no$\star$ \emph{The diffusion term. }Using the explicit expression of $p_i$, the second term of the right hand side
of \fer{logid} (the diffusion term) may be expressed  as $-\tau I^i$, with $I^i$ split as $ I^i= I^i_0 + I^i_1+I^{ik}_2 $, where
\begin{align*}
  & I^i_0 & =  \frac{c_i}{2}\int_\O\int_\O & J_\eps(\bx-\by)  \big(u_i^+(\by)-   u_i^+(\bx)\big)  \big(\ln(u_i^+(\by)+\eps)-\ln(u_i^+(\bx)+\eps) \big)d\by d\bx&\\
  & I^i_1 & =  \frac{a_i}{2} \int_\O\int_\O & J_\eps(\bx-\by) \big(u_i^+(\by)+  u_i^+(\bx)+2\eps\big)  \big(u_i^+(\by)- u_i^+(\bx)\big)  &\\
  &&& \times \big(\ln(u_i^+(\by)+\eps)-\ln(u_i^+(\bx)+\eps) \big)d\by d\bx&\\
 & I^{ik}_2 & = \frac{1}{2}\int_\O\int_\O & J_\eps(\bx-\by)  \big((u_i^+(\by)+\eps)(u_k^+(\by)+\eps)  - (u_i^+(\bx)+\eps)(u_k^+(\bx)+\eps)\big) & \\
  &&& \times  \big(\ln(u_i^+(\by)+\eps)-\ln(u_i^+(\bx)+\eps) \big)d\by d\bx,&   
\end{align*}
for $i,k=1,2$, $i\neq k$. 

The non-negativity of $I^i_0$ and $I^{12}_2  + I^{21}_2$ is directly deduced from the monotonicity of the $\ln$ function. This is straightforward for $I^i_0$. For the cross-diffusion terms, we have, 
\begin{align*}
 I^{12}_2 & + I^{21}_2  = \frac{1}{2}\int_\O\int_\O J_\eps(\bx-\by)  \big((u_1^+(\by)+\eps)(u_2^+(\by)+\eps)  - (u_1^+(\bx)+\eps)(u_2^+(\bx)+\eps)\big) \\
 & \times  \big(\ln\big((u_1^+(\by)+\eps)(u_2^+(\by)+\eps)\big)-\ln\big((u_1^+(\bx)+\eps)(u_2^+(\bx)+\eps)\big) \big)d\by d\bx \geq0.
 \end{align*}
Due to the symmetry of $J_\eps$, the self-diffusion terms may be expressed as 
\begin{align*}
I_1^i = a_i \int_\O\int_\O J_\eps(\bx-\by) \big(u_i^+(\by)+ \eps \big) & \big(u_i^+(\by)- u_i^+(\bx)\big) \\
  & \times \big(\ln(u_i^+(\by)+\eps)-\ln(u_i^+(\bx)+\eps) \big)d\by d\bx.
\end{align*}
Using the elementary inequality
\begin{align}
\label{elemineq1}
s (\ln(s)-\ln(\sigma))\geq s-\sigma \qtext{for all }s,\sigma>0,
\end{align}
we obtain
\begin{align*}
I_1^i \geq a_i \int_\O\int_\O J_\eps(\bx-\by)  \big(u_i^+(\by)- u_i^+(\bx)\big)^2 d\by d\bx.
\end{align*}
This estimate and the non-negativity of $I^i_0$ and $I^{12}_2+I^{21}_2$ imply 
\begin{align}
\label{entropy2}
\sum_{i=1}^2 I^i \geq \sum_{i=1}^2 a_i \int_\O\int_\O J_\eps(\bx-\by)  \big(u_i^+(\by)- u_i^+(\bx)\big)^2 d\by d\bx.
\end{align}

\medskip

\no$\star$ \emph{The Lotka-Volterra term. } We have, for $i,k=1,2$, $i\neq k$, 
\begin{align*}
 \int_\O f_i(\bu^+(\bx)+\eps) & \ln(u_i^+(\bx)+\eps) d\bx = 
\alpha_i  \int_\O (u_i^+(\bx)+\eps) \ln(u_i^+(\bx)+\eps) d\bx \\
& -\beta_{ii}  \int_\O (u_i^+(\bx)+\eps)^2 \ln(u_i^+(\bx)+\eps) d\bx \\
& -\beta_{ik}  \int_\O (u_i^+(\bx)+\eps)(u_k^+(\bx)+\eps) \ln(u_i^+(\bx)+\eps) d\bx = F_0^i+F_1^i+F_2^{ik}.
\end{align*}
The first term may be rewritten as 
\begin{align*}
F_0^i = \alpha_i E_i + \alpha_i\nor{u_i^+}_{L^1}+\alpha_i \abs{\O}\eps.
\end{align*}
For the second term, using that $s^2\ln(s)\geq -\frac{1}{2\e}$ for $s>0$ , we obtain $F_1^i \leq \frac{\beta_{ii}}{2\e}$.
The cross terms are bounded as follows. If $\beta_{12}=\beta_{21}=0$ then we have nothing to do. Assume, without loss of generality, that $\beta_{12}> 0$ and $\beta_{12}>\beta_{21}$, and let $r= \beta_{21}/\beta_{12}$. 
Then, for $s,\sigma>0$, we have
\begin{align*}
\beta_{12}s\sigma\ln(s)+\beta_{21}s\sigma\ln(\sigma) = \beta_{12}\sigma^{1-r}s\sigma^r\ln(s\sigma^r).
\end{align*}
Using the inequality \fer{elemineq1}, we deduce
\begin{align*}
\beta_{12}s\sigma\ln(s)+\beta_{21}s\sigma\ln(\sigma) \geq \beta_{12}\sigma^{1-r}(s\sigma^r-1 ) \geq -\beta_{12}\sigma^{1-r}.
\end{align*}
Therefore, from H\"older's inequality we deduce
\begin{align*}
F_2^{12}+F_2^{21} & \leq \beta_{12} \int_\O (u_2^+(\bx)+\eps)^{1-r} d\bx \leq \beta_{12}\abs{\O}^r \Big(\int_\O (u_2^+(\bx)+\eps) d\bx\Big)^{1-r}\\
& \leq c(1+\nor{u_2^+}_{L^1}),
\end{align*}
were we used $\abs{x}^{1-r}\leq 1+ \abs{x}$, for $r\in(0,1)$.
Gathering the previous estimates yields 
\begin{align}
\sum_{\substack{i,k= 1 \\ k \neq i}}^2  (F_0^i+F_1^i+F_2^{ik}) \leq c \Big( 1 + \sum_{i= 1}^2 \big(E_i  + \nor{u_i^+}_{L^1}\big) \Big). \label{entropy3}
\end{align}
Finally, using \fer{entropy1}, \fer{entropy2} and  \fer{entropy3} in \fer{logid}, we deduce \fer{lema1:entropy}.

\medskip

\no$\bullet$ \emph{$W^{1,\infty}(\O)$ estimate. } Differentiating \fer{eq.eq} with respect to $x_k$, for $k=1,\ldots,d$, we obtain, for $i=1,2$ 
\begin{align*}
 \p_{x_k}u_i(\bx) & = \p_{x_k}u_i^j(\bx)  \\
 & + \tau  \int_\O \p_{x_k} J_\eps(\bx-\by) \big(p_i (\bu^+(\by)+\eps)-p_i (\bu^+(\bx)+\eps)\big) d\by \\
 &-\tau   \Big(\p_1 p_i (\bu^+(\bx)+\eps)\p_{x_k} u_1^+(\bx) 
 + \p_2 p_i (\bu^+(\bx)+\eps)\p_{x_k} u_2^+(\bx) \Big) \int_\O J_\eps(\bx-\by) d\by\\ 
 & + \tau \Big(\p_1 f_i (\bu^+(\bx)+\eps)\p_{x_k} u_1^+(\bx) 
 + \p_2 f_i (\bu^+(\bx)+\eps)\p_{x_k} u_2^+(\bx) \Big).
\end{align*}
Therefore, using \fer{propeps2} and (L), we obtain
\begin{align*}
\abs{ \p_{x_k}u_i(\bx)} & \leq \abs{ \p_{x_k}u_i^j(\bx)}  
+\tau \max(1,J_0) L(\nor{\bu}_{L^\infty})\Big(\nor{\p_{x_k} J_\eps}_{L^1}+\sum_{n=1}^2\p_{x_k} u_n(\bx)  \Big).
\end{align*}
Summing in $i=1,2$, taking the supremum in $k=1,\ldots,d$, and recalling that $\nor{\p_{x_k} J_\eps}_{L^1}$ is uniformly bounded, we deduce \fer{lema1:derivada}.

\medskip

Finally, \fer{lema1:entropy2} is deduced as follows
\begin{align}
\frac{1}{2}\int_\O\int_\O & J_\eps(\bx-\by)  \big(u_i^+(\by)- u_i^+(\bx)\big)^2 d\by d\bx =  \int_\O\int_\O J_\eps(\bx-\by) d\by \abs{u_i^+(\bx)}^2  d\bx \nonumber\\
& - 
\int_\O\int_\O J_\eps(\bx-\by)  u_i^+(\by) u_i^+(\bx)d\by d\bx \geq 
J_0\nor{u_i^+}_{L^2}^2 - \nor{J_\eps}_{L^\infty} \nor{ u_i^+}_{L^1}^2. \label{poin1}
\end{align}
\begin{remark}
Identity \fer{poin1} leads to a nonlocal variant of Poincare's inequality which provides an estimate of the $L^2(\O)$ norm of a function in terms of its $L^1(\O)$ norm and the norm of its nonlocal gradient in $L^2(\O)$. More explicitely, for $v\in L^2(\O)$ and $J$ satisfying (H), we have 
\begin{align*}
\nor{v}_{L^2}^2 \leq \frac{\nor{J}_{L^\infty}}{J_0}  \nor{v}_{L^1}^2 + \frac{1}{2J_0}\int_\O\int_\O J(\bx-\by)  \big(v(\by)- v(\bx)\big)^2 d\by d\bx.
\end{align*}
See \cite{Andreu2010book} for a generalization to $L^q(\O)$.
\end{remark}

\end{proof}

\subsection{Passing to the limit $\tau\to 0$}

Consider the partition of the interval $[0,t_N]$ given by 
$t_0=0$ and $t_j=\sum_{k=1}^j \tau_{k-1}$ for $j=1,\ldots,N$, where $\tau_k$ satisfies 
\fer{cond.tau}.
We define, for $(t,\bx)\in (t_j,t_{j+1}]\times \O$, for $j=0,\ldots,N-1$ the time piecewise constant and
piecewise linear functions given by
\begin{align*}
 u_i^{(\tau)}(t,\bx)=u_i^{j+1}(\bx),\quad  
 \tilde u_i^{(\tau)}(t,\bx)=u_i^{j+1}(\bx)+\frac{t_{j+1}-t}{\tau_j}(u_i^j(\bx)-u_i^{j+1}(\bx)),
\end{align*}
where $(u_1^{j+1},u_2^{j+1})$ is the solution of  \fer{eq:discu}. 
We also consider the shift operator $\sigma_\tau u_i^{(\tau)}(t,\cdot) = u_i^{j}$, for $t\in(t_j,t_{j+1}]$.  With this notation,  equation \fer{eq:discu} may be rewritten as, for $(t,\bx)\in Q_{t_N}$,   
\begin{align}
 \p_t \tilde u_i^{(\tau)}(t,\bx) = &  \int_\O J_\eps(\bx-\by) \big(p_i ((\bu^{(\tau)})^+(t,\by)+\eps)-p_i ((\bu^{(\tau)})^+(t,\bx)+\eps) \big) d\by \label{eq:discuc}\\ 
 & + f_i((\bu^{(\tau)})^+(t,\bx)+\eps). \nonumber
\end{align}

\begin{corollary}
\label{cor.tau}
For $i=1,2$, the norms
\begin{align*}
\nor{\grad u_i^{(\tau)}}_{L^\infty(Q_{t_N})} ,\quad 
 \nor{\grad \tilde u_i^{(\tau)}}_{L^\infty(Q_{t_N})},
\end{align*}
are uniformly bounded with respect to  $\tau$. In addition, for $c$ independent of $\eps$ and $\tau$, 
\begin{align}
\label{cor1}
\nor{u_i^{(\tau)}}_{L^\infty(Q_{t_N})} \leq c,\quad  \nor{(u_i^{(\tau)})^-}_{L^\infty(0,t_N;L^1(\O))} \leq c\eps, \quad 
 \nor{\p_t\tilde u_i^{(\tau)}}_{L^\infty(Q_{t_N})} \leq c,
\end{align}
and, for $t\in [0,t_N]$, 
\begin{align}
  E^{(\tau)}(t)  & +\sum_{i=1}^2 a_i \int_{Q_t}\int_\O J_\eps(\bx-\by)  \big((u_i^{(\tau)})^+(s,\by)- (u_i^{(\tau)})^+(s,\bx)\big)^2 d\by d\bx ds    \nonumber \\
& - \ln(\eps) \sum_{i=1}^2  \nor{(u_i^{(\tau)})^-}_{L^1} \leq 
E^{(\tau)}(0) + c (1+ \eps) , \label{lema2:entropy}
\end{align}
with 
\begin{align*}
  E^{(\tau)}(t) = \sum_{i=1}^2 \int_\O ((u_i^{(\tau)})^+(t,\bx)+\eps)\big(\ln((u_i^{(\tau)})^+(t,\bx)+\eps)-1) d\bx.
\end{align*}

\end{corollary}

\begin{proof}
The result is a straightforward consequence of the estimates obtained in Lemma~\ref{lemma:1} and Gronwall's lemma. For instance, from \fer{lema1:ele1mas} we get, summing on $j=0,\ldots,N-1$,
\begin{align}
\label{ggc1}
& \sum_{i=1}^2 \big(\nor{(u_i^{(\tau)})^+(t_N)}_{L^1(\O)} + \beta_{ii}\nor{(u_i^{(\tau)})^+}_{L^2(Q_{t_N})}^2 \big) \leq  
 \sum_{i=1}^2 \nor{(u_{0i}^{(\tau)})^+}_{L^1(\O)} \\
& \qquad +c \sum_{i=1}^2 \nor{(u_i^{(\tau)})^+}_{L^1(Q_{t_N})}   +c\eps t_N. \nonumber
\end{align}
Gronwall's inequality implies 
\begin{align*}
& \sum_{i=1}^2 \big(\nor{(u_i^{(\tau)})^+(t_N)}_{L^1(\O)} \leq  e^{ct_N} \big( 
 \sum_{i=1}^2 \nor{(u_{0i}^{(\tau)})^+}_{L^1(\O)}  +c\eps t_N \big) \leq c,
\end{align*}
and then from \fer{ggc1} we also get 
\begin{align*}
& \sum_{i=1}^2 \beta_{ii}\nor{(u_i^{(\tau)})^+}_{L^2(Q_{t_N})}^2 \leq  c.
\end{align*}
Similarly, we obtain from \fer{lema1:ele1menos} 
\begin{align*}
\sum_{i=1}^2 \nor{(u_i^{(\tau)})^-(t_N)}_{L^1(\O)} \leq c\eps,
\end{align*}
and then, from \fer{lema1:entropy2}, 
\begin{align*}
\sum_{i=1}^2 a_{i}\nor{(u_i^{(\tau)})^+}_{L^2(Q_{t_N})}^2 \leq  c.
\end{align*}
Being the $L^1(Q_{t_N})$ and the $L^2(Q_{t_N})$ norms of $u_i^{(\tau)}$ uniformly bounded with respect to $\tau$ and $\eps$, the uniform bound for its 
$L^\infty(Q_{t_N})$ norm is then deduced from \fer{lema1:eleinf} and Gronwall's lemma. And then, the uniform bounds with respect to $\tau$ for  the norms of $\grad  u_i^{(\tau)}$,  $\grad \tilde u_i^{(\tau)}$ are deduced from \fer{lema1:derivada} and the uniform bound on $\nor{u_i^{(\tau)}}_{L^\infty(Q_{t_N})}$. Observe that these bounds are not uniform with respect to $\eps$, since they depend on $\nor{\grad u_{0\eps i}}_{L^\infty}$, see \fer{ass:regularity}.  By definition, the norm $\nor{\p_t \tilde u_i^{(\tau)}}_{L^\infty(Q_{t_N})}$ is bounded in terms of $\nor{u_i^{(\tau)}}_{L^\infty(Q_{t_N})}$, and thus uniformly bounded with respect to $\tau$ and $\eps$. Finally, \fer{lema2:entropy} is deduced from the previous estimates and Gronwall's lemma applied to \fer{lema1:entropy}.
\end{proof}

Corollary~\ref{cor.tau} implies the existence of functions
$u_i\in L^\infty(0,{{t_N}}; W^{1,\infty}(\O))$ and 
$\tilde u_i\in W^{1,\infty}(Q_{{t_N}})$ such that, at least for subsequences (not relabeled)
\begin{align}
 & u_i^{(\tau)} \to u_i \qtext{weakly* in }L^\infty(0,{t_N}; W^{1,\infty}(\O)), \nonumber \\
 & \tilde u_i^{(\tau)} \to \tilde u_i \qtext{weakly* in }   W^{1,\infty}(Q_{{t_N}}), \label{conv.1}
\end{align}
as $\tau \to 0$.
In particular, by compactness 
\begin{equation*}
 \tilde u_i^{(\tau)} \to \tilde u_i \qtext{uniformly in } C([0,{t_N}]\times \bar \O). 
\end{equation*}
Since, for $t\in (t_j,t_{j+1}]$,
\begin{align*}
 \abs{ u_i^{(\tau)}(t,\bx) -\tilde u_i^{(\tau)}(t,\bx) }=& \left|\frac{(j+1)\tau -t}{\tau}(u_i^j(\bx)-u_i^{j+1}(\bx))\right| \\
 & \leq \tau \nor{\p_t \tilde u_i^{(\tau)}}_{L^{\infty}(Q_{{t_N}})},
\end{align*}
we deduce both $u_i=\tilde u_i$ and, up to a subsequence, 
\begin{equation}
 \label{conv.2}
 u_i^{(\tau)} \to  u_i \qtext{strongly in } L^\infty(Q_{t_N}) \text{ and a.e. in } Q_{t_N}. 
\end{equation}
With the properties of convergence \fer{conv.1} and \fer{conv.2} the passing to the limit $\tau\to0$ in \fer{eq:discuc} is justified, finding that, for $i=1,2$,  $u_i\in W^{1,\infty}(Q_{t_N})$ is a solution of 
\begin{align}
\label{eq:discuc2}
 \p_t u_i(t,\bx) = &  \int_\O J_\eps(\bx-\by) \big(p_i (\bu^+(t,\by)+\eps)-p_i (\bu^+(t,\bx)+\eps) \big) d\by \\ 
 & \quad + f_i(\bu^+(t,\bx)+\eps), \nonumber \\
 u_{i}(0,\bx) = & ~u_{0i\eps}(\bx).  \label{eq:id2}
\end{align}
Moreover, from \fer{cor1} and \fer{lema2:entropy} we deduce  
\begin{align}
\label{estinff}
\nor{u_i}_{L^\infty(Q_{t_N})} \leq c,\quad  \nor{u_i^-}_{L^\infty(0,t_N;L^1(\O))} \leq c\eps,  \quad 
 \nor{\p_t u_i}_{L^\infty(Q_{t_N})} \leq c,
\end{align}
and, for $t\in [0,t_N]$, 
\begin{align}
  E(t)  & +\sum_{i=1}^2 a_i \int_{Q_t}\int_\O J_\eps(\bx-\by)  \big(u_i^+(s,\by)- u_i^+(s,\bx)\big)^2 d\by d\bx ds   - \ln(\eps) \sum_{i=1}^2  \nor{u_i^-}_{L^1}  \nonumber \\
& \leq 
E(0) + c (1+ \eps) , \label{conc:entropy}
\end{align}
for some constant $c>0$ independent of $\eps$.

Thanks to the $L^\infty(Q_{t_N})$ uniform estimate on $u_i$, we may go back to the fixed point operator 
\fer{fpop} and obtain a sequence of functions satisfying \fer{eq:discu} for the initial iteration $\bu_\eps^0 = \bu_\eps(t_N,\cdot)$. These functions satisfy the estimates of Lemma~\ref{lemma:1}, so we may define from them a solution of \fer{eq:discuc2}-\fer{eq:id2} in the time interval 
$[0,2t_N]$.  This procedure may be continued until reaching any arbitrarily fixed final time, $T$. 

\section{Passing to the limit $\eps\to0$}\label{sec:eps}

Let us denote by $\bu_\eps$ to the solution of \fer{eq:discuc2}-\fer{eq:id2} so that 
\fer{estinff} is rewritten as, for some constant $c>0$ independent of $\eps$, 
\begin{align}
\label{estinff2}
\nor{u_{\eps i}}_{L^\infty(Q_{T})} \leq c,\quad  \nor{u_{\eps i}^-}_{L^\infty(0,T;L^1(\O))} \leq c\eps,
                   \quad 
 \nor{\p_t u_{\eps i}}_{L^\infty(Q_{T})} \leq c.
\end{align}
Being $u_{0\eps}, J_\eps$ smooth functions, we may deduce an $L^\infty(Q_T)$ bound for $\grad u_\eps$ as in \fer{lema1:derivada}, not necessarily uniform in $\eps$, but  allowing to differentiate equation \fer{eq:discuc2} with respect to $x_k$ to obtain,
for $k=1,\ldots,d$,  $i,j=1,2$, $i\neq j$,                   
\begin{align}
\label{ggc2}
  \p_t  \p_{x_k}  u_{\eps i}(t,\bx)  = & \int_\O \p_{x_k} J_\eps(\bx-\by) \big(p_i (\bu_\eps^+(t,\by)+\eps)-p_i (\bu_\eps^+(t,\bx)+\eps)\big) d\by \\
 & + \sum_{j=1}^2 \Big[\Big( \p_jf_i (\bu_\eps^+(t,\bx)+\eps)- m_\eps(\bx) \p_jp_i (\bu_\eps^+(t,\bx)+\eps) \Big) \nonumber \\
 &\qquad \times  \sign(u_{\eps j}(t,\bx)+\eps) \p_{x_k} u_{\eps j}(t,\bx) \Big],\nonumber
\end{align}
with $m_\eps$ given by \fer{conc.Jeps}. 
Identity \fer{ggc2}  may be written in matrix form as 
\begin{align}
\label{edo}
 \p_t \bv_\eps(t,\bx) = \bA_\eps(t,\bx) \bv_\eps(t,\bx) + \bb_\eps(t,\bx),
\end{align}
with $v_{\eps i}(t,\bx) = \p_{x_k}u_{\eps i}(t,\bx),$
\begin{align*}
& A_{\eps ij} (t,\bx)=  \Big( \p_jf_i (\bu_\eps^+(t,\bx)+\eps)- m_\eps(\bx) \p_jp_i (\bu_\eps^+(t,\bx)+\eps) \Big)
 \sign(u_{\eps j}(t,\bx)+\eps),\\
& b_{\eps i} (t,\bx) =  \int_\O \p_{x_k} J_\eps(\bx-\by) \big(p_i (\bu_\eps^+(t,\by)+\eps)-p_i (\bu_\eps^+(t,\bx)+\eps)\big) d\by .
\end{align*}
Since $u_{\eps i}$ is uniformly bounded in $L^\infty(Q_T)$ and $\p_{x_k} J_\eps$ is uniformly bounded in $L^1(B)$ we deduce, using properties  and  \fer{conc.Jeps2} and (L),   
that $A_{\eps ij},b_{\eps i}$ are uniformly bounded in $L^\infty(Q_T)$. Integrating \fer{edo} in $(0,t)$ we obtain 
\begin{align*}
 \p_{x_k}u_{\eps i}(t,\bx) = G_{\eps i}(t,\bx) + H_{\eps i1}(t,\bx)\p_{x_k}u_{0\eps 1}(\bx) + 
 H_{\eps i2}(t,\bx)\p_{x_k}u_{0\eps 2}(\bx),
\end{align*}
with $G_{\eps i},H_{\eps ij}$ uniformly bounded in $L^\infty(Q_T)$. Finally, since $\p_{x_k}u_{0\eps i}$ is uniformly bounded in $L^1(\O)$, see \fer{propeps1b},  we deduce
\begin{align}
\label{bound.2}
\p_{x_k}u_{\eps i} \text{ is uniformly bounded in }L^\infty(0,T;L^1(\O)). 
\end{align}
The time derivative bound in \fer{estinff2} and \fer{bound.2}
allow to deduce, using the compactness result \cite[Cor. 4, p. 85]{Simon1986},  the existence of $u_i\in C([0,T];L^\infty(\O)\cap BV(\O))$ such that $u_{\eps i}\to u_i$ strongly in  $L^q(Q_T)$, for all $q<\infty$, and a.e. in $Q_T$. 
The time derivative uniform bound in \fer{estinff2} also implies that, up to a subsequence (not relabeled), we have 
$\p_t u_{\eps i}\to \p_t u_i$ weakly* in $L^\infty(Q_T)$. 

These convergences allow to pass to the limit $\eps\to 0$ in \fer{eq:discuc2}-\fer{eq:id2} 
(with $u$ replaced by $u_\eps$) and identify the limit 
\begin{align*}
u_i \in W^{1,\infty}(0,T;L^\infty(\O))\cap C([0,T];L^\infty(\O)\cap BV(\O)), 
\end{align*}
 as a solution of \fer{eq.eq}-\fer{eq.id}. 
Observe that, since $u_{\eps i}$ satisfies the second bound of \fer{estinff2} we also deduce that  $u_i\geq 0$ a.e. in $Q_T$. Finally, the strong and a.e. convergences of $\bu_\eps$ and $J_\eps$ also allow to pass to the limit in \fer{conc:entropy} to deduce \fer{th1:entropy}.

\begin{remark}
Observe that, like in the local diffusion problem, the non-negativity of the limit solution may also be deduced from the entropy inequality \fer{conc:entropy}.
\end{remark}

\section{Uniqueness of solution}\label{sec:uniqueness}
We use a duality technique to prove the uniqueness of solution. Let $\bu,~\bv$ be two solutions of \fer{eq.eq}-\fer{eq.id} and set $\bw = \bu-\bv$. Then, for $i=1,2$ and $(t,\bx)\in Q_T$, we have $w_i(0,\bx)=  0$ and 
\begin{align*}
\p_t w_i(t,\bx) =&  \int_\O J(\bx-\by) \Big(p_i(\bu(t,\by)-p_i(\bv(t,\by))- \big(p_i(\bu(t,\bx)-p_i(\bv(t,\bx))\big)\Big) d\by\\
& + f_i(\bu(t,\bx)-f_i(\bv(t,\bx))\big).
\end{align*}
Testing this equation with some $\vfi_i \in W^{1,1}(0,T;L^1(\O))$, we obtain, for $i,j=1,2$ and $i\neq j$ ,
\begin{align*}
\int_\O \p_t w_i(t,\bx) & \vfi_i(t,\bx) d\bx  \\
=& -\int_\O \int_\O J(\bx-\by) \big(p_i(\bu(t,\bx))-p_i(\bv(t,\bx))\big) \big(\vfi_i(t,\by)-\vfi_i(t,\bx)\big) d\by d\bx \\
& + \int_\O \big(f_i(\bu(t,\bx)-f_i(\bv(t,\bx))\big) \vfi_i(t,\bx) d\bx.
\end{align*}
Integrating in $(0,T)$,  imposing $\vfi_i(T,\bx)=0$, and using that $w_i(0,\bx)=0$ and the explicit form of $\bp$ and $\bef$, we get 
\begin{align}
\int_{Q_T} w_i(t,\bx) & \p_t  \vfi_i(t,\bx) d\bx dt  \nonumber \\
=& -\int_{Q_T} w_i(t,\bx) \int_\O J(\bx-\by) K_{ij}(t,\bx) \big(\vfi_i(t,\by)-\vfi_i(t,\bx)\big) d\by d\bx dt \nonumber \\
& -\int_{Q_T} w_j(t,\bx) \int_\O J(\bx-\by) v_i(t,\bx) \big(\vfi_i(t,\by)-\vfi_i(t,\bx)\big) d\by d\bx dt \nonumber\\
& + \int_{Q_T} w_i(t,\bx) L_{ij}(t,\bx)  \vfi_i(t,\bx) d\bx dt+ \int_{Q_T} w_j(t,\bx) \beta_{ij} v_i(t,\bx)  \vfi_i(t,\bx) d\bx dt,\label{uniq.1}
\end{align}
where $K_{ij}= c_i + a_i(u_1+u_2)+u_j$, and $L_{ij}= \alpha_i -\beta_{ii}(u_1+u_2)-\beta_{ij}u_j$.
We introduce the change of time variable $t\to T-t$ and consider the following coupled linear problem: For $i,j=1,2$ and $i\neq j$, find $\vfi_i \in L^\infty(0,T;L^1(\O))$ such that for $(t,\bx)\in Q_T$,
\begin{align}
\label{eq.dual}
 \p_t  \vfi_i(t,\bx)  =   w_i(t,\bx) & +\int_\O J(\bx-\by) K_{ij}(t,\bx) \big(\vfi_i(t,\by)-\vfi_i(t,\bx)\big) d\by \\
& +\int_\O J(\bx-\by) v_j(t,\bx) \big(\vfi_j(t,\by)-\vfi_j(t,\bx)\big) d\by \nonumber\\
& - L_{ij}(t,\bx)  \vfi_i(t,\bx) - \beta_{ji} v_j(t,\bx)  \vfi_j(t,\bx) ,\nonumber\\
 \vfi_i(0,\bx)=0. \qquad &  \label{id.dual}
\end{align}
Observe that if this problem has a solution then, summing \fer{uniq.1} for $i=1,2$ (and recalling the change of time variable), we obtain
\begin{align*}
\sum_{i=1}^2\int_{Q_T} \abs{w_i(t,\bx)}^2 d\bx dt  = 0,
\end{align*}
implying $w_i = 0$ a.e. in $Q_T$, and therefore proving the uniqueness of solution of \fer{eq.eq}-\fer{eq.id}.
The existence of solutions of the linear problem for the test functions may be proved by Banach's fixed point theorem. Let $T_0\in(0,T]$ be a constant to be fixed and consider the Banach space $X_{T_0} = L^\infty(0,T_0;L^\infty(\O))$. We define the operator $\bG=(G_1,G_2)$ in $X_{T_0}\times X_{T_0}$ by, for $(t,\bx)\in Q_{T_0}$, $i,j=1,2$ with $i\neq j$,  
\begin{align*}
G_i(\bpsi) (t,\bx)  = &  \int_0^t w_i(t,\bx)dt +\int_0^t \int_\O J(\bx-\by) K_{ij}(t,\bx) \big(\psi_i(t,\by)-\psi_i(t,\bx)\big) d\by dt\\
& +\int_0^t\int_\O J(\bx-\by) v_j(t,\bx) \big(\psi_j(t,\by)-\psi_j(t,\bx)\big) d\by dt\\
& - \int_0^t \big(L_{ij}(t,\bx)  \psi_i(t,\bx) + \beta_{ji} v_j(t,\bx)  \psi_j(t,\bx) \big) dt.
\end{align*}
 Since solutions of \fer{eq.eq}-\fer{eq.id} are $L^\infty(Q_T)$ functions, we have $K_{ij}, ~L_{ij},~v_i \in L^\infty(Q_{T_0})$. Therefore
\begin{align*}
\nor{\bG(\bpsi)}_{X_{T_0}} = \sum_{i=1}^2 \nor{G_i(\bpsi)}_{X_{T_0}} \leq  & c T_0 (1+\nor{\bpsi}_{X_{T_0}}), 
\end{align*}
with $c$ depending on the $L^\infty(Q_T)$ norms of $\bu,~\bv$, and where, abusing on notation, we denote by $\nor{\cdot}_{X_{T_0}}$ both to the norms of scalar and vector functions. Therefore, $\bG(X_{T_0}\times X_{T_0}) \subset X_{T_0}\times X_{T_0}$.

To prove the contractivity, let $\bpsi,\bxi \in X_{T_0}\times X_{T_0}$. We have, for $(t,\bx)\in Q_{T_0}$, $i,j=1,2$ with $i\neq j$, 
\begin{align*}
G_i(\bpsi) (t,\bx) & - G_i(\bxi) (t,\bx)  \\
= &  \int_0^t \int_\O J(\bx-\by) K_{ij}(t,\bx) \big(\psi_i(t,\by)- \xi_i(t,\by) -(\psi_i(t,\bx)-\xi_i(t,\bx))\big) d\by dt\\
& +\int_0^t\int_\O J(\bx-\by) v_j(t,\bx) \big(\psi_j(t,\by)- \xi_j(t,\by) -(\psi_j(t,\bx)-\xi_j(t,\bx))\big)d\by dt\\
& - \int_0^t \big(L_{ij}(t,\bx)  (\psi_i(t,\bx)-\xi_i(t,\bx)) + \beta_{ji} v_j(t,\bx)  (\psi_j(t,\bx)-\xi_j(t,\bx)) \big) dt.
\end{align*}
Thus, for some $c$ depending on the $L^\infty(Q_T)$ norms of $\bu$ and $\bv$, we deduce
\begin{align*}
\nor{\bG(\bpsi) - \bG(\bxi)}_{X_{T_0}} \leq  cT_0 \nor{\bpsi -\bxi}_{X_{T_0}} .
\end{align*}
Choosing $T_0<1/c$ we obtain that $\bG$ is a strict contraction on $X_{T_0}\times X_{T_0}$.  This proves the existence of a local in time solution of the dual problem \fer{eq.dual}-\fer{id.dual} in the time interval $[0,T_0]$. We may easily extend this solution to any arbitrary $T>0$ by matching solutions in the intervals $[0,T_0]$, $[T_0,2T_0]$, etc. Hence the uniqueness of solution of problem \fer{eq.eq}-\fer{eq.id} follows.


\end{document}